\documentclass[12pt,a4paper]{amsart}

\usepackage{amsmath}
\usepackage{amsfonts}
\usepackage{amssymb}
\usepackage{verbatim}
\usepackage{amsthm}
\usepackage{color}
\usepackage{xypic,mathtools}
\usepackage{hyperref}
\usepackage[left=2cm,right=2cm,top=2cm,bottom=2cm]{geometry}

\theoremstyle{plain}
 \newtheorem{theorem}{Theorem}
 \newtheorem{lemma}{Lemma}
 \newtheorem{corollary}{Corollary}

\theoremstyle{definition}

\DeclareMathOperator{\Aut}{Aut}

\DeclareMathOperator{\Coker}{Coker}

\DeclareMathOperator{\disc}{disc}

\DeclareMathOperator{\Gal}{Gal}

\DeclareMathOperator{\GL}{GL}

\DeclareMathOperator{\image}{image}

\DeclareMathOperator{\lcm}{lcm}

\DeclareMathOperator{\Pic}{Pic}

\DeclareMathOperator{\SL}{SL}

\newcommand{\IAf}{\mathbb{A}_{\IQ,f}}

\newcommand{\IZ}{\mathbb{Z}}

\newcommand{\IZhat}{\widehat{\IZ}}

\newcommand{\IQ}{\mathbb{Q}}

\newcommand{\IQbar}{\overline{\mathbb{Q}}}
\newcommand{\IR}{\mathbb{R}}
\newcommand{\IC}{\mathbb{C}}
\newcommand{\IF}{\mathbb{F}}

\newcommand{\Gm}{\mathbb{G}_m}

\newcommand{\IA}{\mathbb{A}}

\newcommand{\pr}{\textrm{pr}}

\title{Intersections of Class Fields}

\author{Lars K\"uhne}
\email{lars.kuehne@unibas.ch}
\address{Departement Mathematik und Informatik \\
Spiegelgasse 1 \\
4051 Basel \\
Switzerland}



\begin{document}

\maketitle

\begin{abstract}
Using class field theory, we prove a restriction on the intersection of the maximal abelian extensions associated with different number fields. This restriction is then used to improve a result of Rosen and Silverman about the linear independence of Heegner points. In addition, it yields effective restrictions for the special points lying on an algebraic subvariety in a product of modular curves. The latter application is related to the Andr\'e-Oort conjecture.
\end{abstract}

\section{Introduction}

The arithmetic of modular curves is essentially determined by their special points. In fact, their canonical models are defined by demanding the ``right'' Galois action on their special points. This Galois action can be expressed completely in terms of class field theory. This is why it is not surprising that problems in the arithmetic geometry of Shimura curves, say the Andr\'e-Oort conjecture for $Y(1)^n$, can be approached by studying the class field theory associated with special points. In particular, intersections of ring class fields appear rather frequently in this situation (e.g., in \cite{Andre1998} and \cite{Rosen2007}).

Our first result is a rather general restriction on the intersection of the maximal abelian extensions of distinct number fields. Throughout this article, we consider all number fields as subfields of a fixed algebraic closure $\IQbar$. In addition, we denote by $K^{\mathrm{ab}} \subset \IQbar$ the maximal abelian extension of a number field $K \subset \IQbar$. We say that a positive integer $n$ annihilates a profinite group $G$ if $g^n=1$ for all $g \in G$. In other words, $n$ is said to annihilate $G$ if the profinite group $G$ has a finite exponent $e(G)$ and $n$ divides $e(G)$.

\begin{theorem} \label{theorem::rayclassfields}
Let $k$ be a number field and let $K_1 / k, \dots, K_r / k$ ($r\geq 2$) be finite normal extensions such that $K_i \cap K_j = k$. Write $K = K_1K_2\cdots K_r$ and $L = K_r \prod_{i=1}^{r-1}K_i^{\mathrm{ab}}$. Then, the Galois group $\Gal(K_r^{\mathrm{ab}} \cap L/K_r k^{\mathrm{ab}})$ is annihilated by 
\begin{equation} \label{equation::n}
e(\Gal(K k^{\mathrm{ab}} /K_rk^{\mathrm{ab}})) \cdot \lcm_{i \neq j}\{[K:K_i][K_j:k]\}.
\end{equation}
\end{theorem}

This simple theorem does not seem to have appeared in the literature so far. 
We deduce the following corollary concerning ring class fields. To state it, we briefly indicate our notation but postpone details to Sections \ref{section::classfieldtheory} and \ref{subsection::ringclassfields}. For a CM-field $K \subset \IQbar$ with maximal totally real subfield $F$, there is a group-theoretic transfer map $\Gal(F^{\mathrm{ab}}/F) \rightarrow \Gal(K^{\mathrm{ab}}/K)$. The fixed field of its image, which we call henceforth the ``transfer field'' $K^{\mathrm{tf}} \subset \IQbar$, is closely connected to ring class fields. In fact, let $K^{\mathrm{rcf}} \subset \IQbar$ denote the union of all ring class fields associated with orders in $K$. It is then easy to prove that $\Gal(K^\mathrm{rcf}/K^\mathrm{tf}) \approx \Pic(\mathcal{O}_F)$ (Lemma \ref{lemma::transferfield}). In particular, both fields coincide for $F=\IQ$.

\begin{corollary} \label{corollary::ringclassfields} Given a totally real field $F$, let $K_1, \dots, K_r$ ($r \geq 2$) be distinct CM-fields with maximal totally real subfield $F$. Write $L = K_r \prod_{i=1}^{r-1}K_i^{\mathrm{tf}}$. Then, the Galois group $\Gal(K_r^{\mathrm{tf}} \cap L/K_r)$ is annihilated by $2^{r+1}$.
\end{corollary}

This result is an improvement of \cite[Proposition 18]{Rosen2007}. For the intersection of two ring class fields (i.e., $r=2$), a slightly stronger result (cf.\ \cite[Theorem 8.3.12]{Cohn1985}) is already known. This more precise result of Cohn, with $2^{r-1}$ instead of $2^{r+1}$, does not follow from our more general one. However, this numerical weakness causes probably no harm in most applications.

Having established Corollary \ref{corollary::ringclassfields}, we demonstrate its efficacy by describing its application to two problems of arithmetic geometry concerning modular curves. In these applications, there is neither need for the generality provided by Theorem \ref{theorem::rayclassfields} nor for Corollary \ref{corollary::ringclassfields} in the case $F \neq \IQ$. However, the exposition does not simplify by dropping either generality. In addition, Theorem \ref{theorem::rayclassfields} seems to merit record as it may have some uses quite distant from those developed here.

Our first application of Corollary \ref{corollary::ringclassfields} yields a strengthening of a result on Heegner points proven by Rosen and Silverman \cite{Rosen2007}. We refer to Section \ref{subsection::preliminariesmodularcurves} for a precise definition of the modular curves $Y_0(N)$, $N \geq 1$, and to Section \ref{subsection::cmpoints} for a discussion of CM-points. The Jacobian of the compactification $X_0(N)$ of $Y_0(N)$ is denoted by $J_0(N)$. By the work of Wiles, Taylor, and others (\cite{Breuil2001, Taylor1995, Wiles1995}), each elliptic curve $E/\IQ$ of conductor $N$ admits a modular parameterization by $X_0(N)$ (i.e., a surjective homomorphism $J_0(N) \rightarrow E$ of abelian varieties).

\begin{theorem} \label{theorem::heegner} Let $E/\mathbb{Q}$ be an elliptic curve without complex multiplication and $\pi: J_0(N) \rightarrow E$ a modular parameterization. Then, there exists an effectively computable constant $c_1=c_1(E,\pi,r)>0$ for each positive integer $r$ such that the following assertion is true. Let $P_i$, $1\leq i \leq r$, be a CM-point on $Y_0(N)$ with associated CM-field $K_i$. One of the following three statements is true:
\begin{enumerate}
\item $K_i = K_j$ for some $1 \leq i < j \leq r$, or
\item one of the CM-points $P_i$ is associated with an order of discriminant less than $c_1$, or
\item the points $\pi(P_1), \pi(P_2), \dots, \pi(P_r) \in E(\IQbar)$ generate a free $\IZ$-module of rank $r$.
\end{enumerate}
\end{theorem}

In the original result of Rosen and Silverman, a class of CM-fields has to be excluded. Heuristically, this class should contain infinitely many CM-fields (cf.\ \cite[Section 10]{Rosen2007}). Replacing their central \cite[Proposition 18]{Rosen2007} with our stronger Corollary \ref{corollary::ringclassfields}, we can lift this restriction effectively. To obtain an effective constant $c_1$, we also have to replace Siegel's class number theorem \cite{Siegel1935} with the Siegel-Tatuzawa theorem \cite{Tatuzawa1951} as well as to use Lombardo's effective version \cite{Lombardo2015} of Serre's open image theorem \cite{Serre1972}. The Siegel-Tatuzawa theorem has already been used in a previous work of the author \cite{Kuehne2013} for similar purposes. For comparison with \cite{Rosen2007}, it should be also mentioned that we impose no Heegner condition on the orders associated with the CM-points $P_i$ (cf.\ \cite[Section 1]{Rosen2007}). This means that the cyclic isogeny encoded within $P_i \in Y_0(N)(\IQbar)$ is not required to constitute an endomorphism of the associated CM-elliptic curve. 

Our second application of Corollary \ref{corollary::ringclassfields} is a weak but effective result of Andr\'e-Oort type for a product of modular curves. Recall that the components of a special point in $Y(1)^n$ are CM-points in $Y(1)$. Consequently, we can associate with each special point $P \in Y(1)^n$ a $n$-tuple $(K_1(P),\dots,K_n(P))$ of CM-fields. We refer to Section \ref{subsection::preliminariesmodularcurves} for further details. To spare additional notation, we do not fix any specific degree or height for subvarieties in $Y(1)^n$. Of course, every common choice works as we do only assert effective boundedness.

\begin{theorem} \label{theorem::andreoort} Let $X$ be an irreducible subvariety of $Y(1)^n$ ($n\geq 2$). Assume that there exists no coordinate permutation $\pi: Y(1)^n \rightarrow Y(1)^n$ and no algebraic subvariety $X^\prime \subset Y(1)^{n-1}$ such that $\pi(X)= Y(1) \times X^\prime$ or $\pi(X) = Q \times X^\prime$ with some point $Q \in Y(1)(\IQbar)$. Then, there exists a proper algebraic subset $Z \subsetneq X$ such that any special point $P$ in $X \setminus Z$ satisfies $K_i(P) = K_j(P)$ for some pair $(i,j)$ with $1 \leq i \neq j \leq n$. The degree, the height, and the degree of the field of definition of $Z$ can be effectively bounded in terms of the same quantities associated with $X$.
\end{theorem}

The theorem generalizes basically the first step of Andr\'e's proof \cite{Andre1998} of the Andr\'e-Oort conjecture for a product of two modular curves. Evidently, a non-effective version of the above theorem is also implied by Pila's Theorem \cite{Pila2011a}. However, his theorem cannot be used to give effective bounds due to several non-effective arguments in his proof. In forthcoming work with Yuri Bilu \cite{Bilu2017}, the proof of Theorem \ref{theorem::andreoort} is refined to prove the Andr\'e-Oort conjecture \textit{effectively} for all linear subvarieties of $Y(1)^n$. This is substantially beyond the currently known cases concerning curves in $Y(1)^2$ \cite{Allombert2015, Bilu2016, Bilu2013, Kuehne2013}. The above Corollary \ref{corollary::ringclassfields} is an essential ingredient for this approach.

\section{Preliminaries}

\subsection{Class Field Theory} \label{section::classfieldtheory}

In order to fix notations, we summarize the basics of class field theory in this section. Our reference is \cite{Neukirch2013}. For each number field $K$, we denote by $I_K$ its idele group, by $C_K=I_K/K^\times$ its idele class group, and by $K^{\mathrm{ab}}$ its maximal abelian extension (in $\IQbar$). For an extension $L/K$ of number fields, there is a canonical inclusion $I_K \hookrightarrow I_L$ and this induces an inclusion $C_K \hookrightarrow C_L$ (\cite[Proposition III.2.6]{Neukirch2013}). In the converse direction, the idele norm $N_{L/K}: I_L \rightarrow I_K$ induces a norm homomorphism $N_{L/K}: C_L \rightarrow C_K$.

Write $G_{\IQ}$ for the (profinite) Galois group $\Gal(\IQbar/\IQ)$. For each number field $K$, every element $\sigma \in G_{\IQ}$ induces (continuous) isomorphisms $I_K \rightarrow I_{\sigma(K)}$ and $C_K \rightarrow C_{\sigma(K)}$. If $L/K$ is normal with Galois group $G=\Gal(L/K)$ then $I_K=I_L^G$ (\cite[Proposition III.2.5]{Neukirch2013}) and $C_K=C_L^G$ (\cite[Theorem III.2.7]{Neukirch2013}). The central object of global class field theory is the topological $G_{\IQ}$-module 
\begin{equation*}
C = \varinjlim_{K}{C_K},
\end{equation*}
where the direct limit is taken over all number fields $K$ with respect to the above-mentioned inclusions. Likewise, the $G_{\IQ}$-action on $C$ is induced from the above-mentioned isomorphisms $C_K \rightarrow C_{\sigma( K)}$. 

Global class field theory uses the fact that the pair $(G_{\IQ},C)$ is a class formation (see \cite[Definition II.1.3]{Neukirch2013}). With this, we have for each finite extension $L/K$ of number fields a reciprocity isomorphism (see \cite[Theorem II.1.9]{Neukirch2013})
\begin{equation*}
\xymatrix{C_K/N_{L|K}C_L \ar[r]^-{\sim} & \Gal(L/K)^{\mathrm{ab}}}
\end{equation*}
and a corresponding norm residue symbol
\begin{equation*}
(\cdot, L|K): \xymatrix{C_K \ar[r] & \Gal(L/K)^{\mathrm{ab}},} 
\xymatrix{ a \ar@{|->}[r] & (a, L|K).}
\end{equation*}
Taking profinite limits, we obtain the universal norm residue symbol
\begin{equation*}
(\cdot, K^{ab}|K): \xymatrix{C_K \ar[r] & \Gal(K^{ab}/K)=\Gal(\IQbar/K)^{\mathrm{ab}}} 
\end{equation*}
for each number field $K$; its kernel is 
\begin{equation*}
D_K=\bigcap_{\substack{K \subseteq L \subset \IQbar \\ [L:K]<\infty}} N_{L|K}C_L \subset C_K.
\end{equation*}
We need some more information on $D_K$ in case $K$ is a CM-field with maximal totally real subfield $F$. Inspecting \cite[Section IX.1]{Artin2009} (or \cite[Section 8.2]{Neukirch2008}), we note that the kernel of $(\cdot,K^{\mathrm{ab}}/K)$ equals
\begin{equation*}
D_K= (\IC^\times)^{[F:\IQ]} \overline{\mathcal{O}_K^\times}K^\times/K^\times \subset C_K
\end{equation*}
where $\overline{\mathcal{O}_K^\times}$ is the topological closure of $\mathcal{O}_K^\times$ in $\widehat{\mathcal{O}}_K^\times$. By Dirchlet's unit theorem, we have furthermore $[\mathcal{O}_K^\times:\mathcal{O}_F^\times]<\infty$. We infer that $\overline{\mathcal{O}_K^\times} K^\times = \overline{\mathcal{O}_F^\times} K^\times \subset I_{K,f}$ and hence
\begin{equation} \label{equation::artinkernel}
D_K \cap {I_{K,f}K^\times/K^\times} = \overline{\mathcal{O}_K^\times} K^\times / K^\times = \overline{\mathcal{O}_F^\times} F^\times / F^\times.
\end{equation}

The reciprocity isomorphism and the norm residue symbol satisfy various functoriality compatibilities that can be found in \cite[Theorem II.1.11]{Neukirch2013}. We state the three most important compatibilities for later use:
\begin{itemize}
\item Let $L/K$ be an extension of number fields. Denote by $\iota_{L/K}: \Gal(L^{ab}/L) \rightarrow \Gal(K^{ab}/K)$ the homomorphism induced from the inclusion $\Gal(\IQbar/L) \subseteq \Gal(\IQbar/K)$. Then, the following diagram is commutative:
\begin{equation} \label{equation::classfieldtheory1}
\xymatrix{ C_L \ar@{->>}[rr]^-{(\cdot, L^{ab}|L)} \ar[d]_-{N_{L|K}} & & \Gal(L^{ab}/L) \ar[d]^{\iota_{L/K}} \\ C_K \ar@{->>}[rr]^-{(\cdot, K^{ab}|K)} & & \Gal(K^{ab}/K).}
\end{equation}
\item Let $L/K$ be an extension of number fields. Denote by $\mathrm{Ver}: \Gal(K^{ab}/K) \rightarrow \Gal(L^{ab}/L)$ the group-theoretic transfer (defined as ``Verlagerung'' in \cite[Definition I.4.10]{Neukirch2013}). Then, the following diagram is commutative:
\begin{equation}
\label{equation::classfieldtheory2}
\xymatrix{ C_K \ar@{->>}[rr]^-{(\cdot, K^{ab}|K)} \ar@{^{(}->}[d]_-{\mathrm{inclusion}} & & \Gal(K^{ab}/K) \ar[d]^{\mathrm{Ver}} \\ C_L \ar@{->>}[rr]^-{(\cdot, L^{ab}|L)} & & \Gal(L^{ab}/L).}
\end{equation}
\item Let $L/K$ be a normal extension of number fields and $\sigma \in \Gal(L/K)$. Then, the following diagram is commutative:
\begin{equation}
\label{equation::classfieldtheory3}
\xymatrix{ C_L \ar@{->>}[rr]^-{(\cdot, L^{ab}|L)} \ar@{->}[d]_-{c \; \mapsto \; \sigma (c)} & & \Gal(L^{ab}/L) \ar[d]^{\tau \; \mapsto \; \sigma \circ \tau \circ \sigma^{-1}} \\ C_L \ar@{->>}[rr]^-{(\cdot, L^{ab}|L)} & & \Gal(L^{ab}/L).}
\end{equation}
\end{itemize}

Finally, we make a simple observation invoking (\ref{equation::classfieldtheory1}): If $L/K$ is an extension of number fields then the preimage of $\Gal(L^{ab}/LK^{ab})$ under $(\cdot, L^{ab}|L)$ is $N_{L/K}^{-1}(D_K)$.
 
\subsection{Modular curves}
\label{subsection::preliminariesmodularcurves} We introduce the modular curve $Y(1)$ as well as the modular curves $Y_0(N)$, $N \geq 1$. The reader may find all necessary details in \cite{Diamond2005, Shimura1971}. However, it is convenient for us to recast the definition of modular curves in Deligne's terminology \cite{Deligne1971, Deligne1979} and we refer to \cite{Milne2003} for details. The advantage of this terminology is that it is interwoven with class field theory. We also keep an embedding $\IQbar \hookrightarrow \IC$ fixed when working with modular curves. Write $\mathbb{A}_{K,f}$ (resp.\ $I_{K,f}$) for the finite adeles (resp.\ ideles) of a number field $K$. We recall that $\mathrm{GL}_2(\IR)$ acts on $\mathcal{H}^{\pm}= \{ z \in \mathbb{C} \ | \ \mathrm{Im}(z) \neq 0\}$ by fractional linear transformations, inducing a $\GL_2(\IQ)$-action on $\mathcal{H}^{\pm}$. 

Let $\mathcal{K}$ be a compact open subgroup of $\GL_2(\IAf)$. An element $g^\prime \in \GL_2(\IQ)$ (resp.\ $k \in \mathcal{K}$) acts on $\mathcal{H}^{\pm} \times \GL_2(\IAf)$ by sending $(x,g)$ to $(g^\prime x,g^\prime g)$ (resp.\ to $(x,gk)$). From these actions, we obtain a double quotient set
\begin{equation*}
S_{\mathcal{K}}(\IC) = \GL_2(\IQ) \backslash \mathcal{H}^{\pm} \times \GL_2(\IAf) / \mathcal{K}.
\end{equation*}
As usual, we write $[x,g]_\mathcal{K}$ for the element of $S_\mathcal{K}(\IC)$ that is represented by $(x,g) \in \mathcal{H}^{\pm} \times \GL_2(\IAf)$. This set can be interpreted as the complex points of its canonical model $S_\mathcal{K}$, which is a smooth algebraic curve defined over $\IQ$. The curve $S_\mathcal{K}$ is not geometrically connected in general. In fact, the determinant $\det: \GL_2 \rightarrow \Gm$ and the sign map $\mathcal{H}^{\pm} \rightarrow \{ \pm \}$, $\tau \mapsto \mathrm{sign}(\mathrm{Im}(\tau))$, induce a map
\begin{equation*}
\GL_2(\IQ) \backslash \mathcal{H}^{\pm} \times \GL_2(\IAf) / \mathcal{K} \longrightarrow \IQ^\times \backslash \{ \pm 1\} \times I_{\IQ,f} / \det(\mathcal{K})
\end{equation*}
and the fibers of this map are the connected components of $S_\mathcal{K}(\IC)$. In other words, the algebraic curve $S_{\mathcal{K}} \times_\IQ \IQbar$ decomposes as a disjoint union of algebraic curves $S_{\mathcal{K},\alpha}$ indexed by $\alpha \in \IQ^\times \backslash \{ \pm 1 \} \times I_{\IQ,f} / \det(\mathcal{K})$. The curves $S_{\mathcal{K},\alpha}$ are not definable over $\IQ$ in general, but it is easy to describe how the Galois action permutes them: Let $\sigma \in \Gal(\IQbar/\IQ)$ and $a \in I_\IQ$ be such that $\sigma|_{\IQ^{\mathrm{ab}}}=(a,\IQ^{\mathrm{ab}}|\IQ)$. Furthermore, let $(\alpha_\infty,\alpha_f) \in \{ \pm 1\} \times I_{\IQ,f}$ be a representative of $\alpha$ and write $a=(a_\infty,a_f) \in \IR^\times \times I_{\IQ,f}$. With these notations, we have $\sigma(S_{\mathcal{K},\alpha})=S_{\mathcal{K},\alpha^\prime}$ with $\alpha^\prime$ being represented by $(\mathrm{sign}(a_\infty) \alpha_\infty, a_f \alpha_f)$ (see \cite[p.\ 349]{Milne2005}). Given $(\alpha_\infty,\alpha_f) \in \{ \pm 1 \} \times I_{\IQ,f}$, we choose an arbitrary $\beta_f \in \GL_2(\IAf)$ such that $\det(\beta_f)=\alpha_f$ and consider the map
\begin{equation*} \label{equation::connectedcomponent}
\mathcal{H}^+ \longrightarrow S_\mathcal{K}(\IC), \ \tau \longmapsto [\alpha_\infty \tau,\beta_f]_\mathcal{K},
\end{equation*}
where $\mathcal{H}^+ = \{ z \in \mathbb{C} \ | \ \mathrm{Im}(z) > 0\}$.
This yields a holomorphic covering of $S_{\mathcal{K},\alpha}(\IC)$, inducing an isomorphism $\Gamma \backslash \mathcal{H}^+ \rightarrow S_{\mathcal{K},\alpha}(\IC)$ with $\Gamma = \SL_2(\IQ) \cap \beta_f \mathcal{K} \beta_f^{-1}$ (cf.\ \cite[Proposition 2.7]{Milne2003}).

We can now specialize the above observations; both $Y(1)$ and $Y_0(N)$ are geometrically connected Shimura curves $S_\mathcal{K}$ (cf.\ \cite[Theorem 7.6.3]{Diamond2005}). To obtain $Y(1)$, we set $\mathcal{K}= \GL_2(\IZhat)$. This implies $\det(\mathcal{K})=\IZhat$ and
\begin{equation*}
\IQ^\times \backslash \{ \pm 1 \} \times I_{\IQ,f} / \det(\mathcal{K}) \approx \IZhat / \IZhat
\end{equation*}
has a single element (cf.\ \cite[Lemma 2.3]{Milne2003}). Thus, $S_\mathcal{K}$ is a geometrically connected $\IQ$-curve and $S_\mathcal{K}(\IC) \approx \SL_2(\IZ) \backslash \mathcal{H}^+$. For $Y_0(N)$, we set
\begin{equation*}
\mathcal{K}=\widehat{\Gamma}_0(N)=\left\{ \begin{pmatrix} a & b \\ c & d \end{pmatrix} \in \GL_2(\IZhat) \mid c \equiv 0  \ (\text{mod $N$})\ \right\}.
\end{equation*}
Again, we have evidently $\det(\mathcal{K})=\IZhat$ so that the above argument also shows that $S_\mathcal{K}$ is a geometrically connected $\IQ$-curve. We also note that $S_\mathcal{K}(\IC) \approx \Gamma_0(N)\backslash \mathcal{H}^+$ by the above, establishing a link to the classical definition of modular curves.

\subsection{CM-points} \label{subsection::cmpoints} Let $K$ be an imaginary quadratic number field. Additionally, fix an isomorphism $\varphi: K \rightarrow \IQ^2$. This induces a $\IQ$-algebra embedding $\rho: K \hookrightarrow \mathrm{Mat}_{2\times 2}(\IQ)$. Note that if $\varphi$ runs through all isomorphisms $K \rightarrow \IQ^2$, the associated embedding $\rho$ runs through the unique $\GL_2(\IQ)$-conjugacy class of embeddings $K \hookrightarrow \mathrm{Mat}_{2 \times 2}(\IQ)$ (cf.\ the Skolem-Noether theorem \cite[Theorem 3.3.14]{Farb1993}). Tensoring $\rho$ with $\IAf
$, we obtain a homomorphism $\IA_{K,f} \hookrightarrow \mathrm{Mat}_{2\times 2}(\IA_{\IQ,f})$, which we also denote by $\rho$ to keep our notations simple. Clearly, we have $\rho(I_{K,f}) \subset \GL_2(\IA_{\IQ,f})$. There also exists a unique $x_\rho \in \mathcal{H}^+ \cap K$ whose stabilizer in $\GL_2(\IQ)$ is $\rho(K^\times)$. The CM-points with CM-field $K$ are the elements of the set
\begin{equation*}
Y_{K}(\IC) = \{ [x_\rho,g]_\mathcal{K} \ | \ g \in \GL_2(\IA_{\IQ,f}) \}\subset S_\mathcal{K}(\IC);
\end{equation*}
this definition does not depend on $\varphi$. By Shimura's construction of $S_\mathcal{K}$, the points in $Y_K(\IC)$ are defined over the maximal abelian extension $K^{\mathrm{ab}}$ and
\begin{equation*}
(t, K^{\mathrm{ab}}/K) [x_\rho,g]_\mathcal{K} = [x_\rho, \rho(t) g]_\mathcal{K}
\end{equation*}
for every $[x_\rho, g]_\mathcal{K} \in Y_K(\IC)$ and every $t \in I_{K,f}$ (cf.\ \cite[Definition 12.8]{Milne2005}).\footnote{Note that all ideles in $\IC^\times \times \{1\} \subset \IC^\times \times I_{f,K}$ are norms with respect to any finite extension $L/K$. This implies $(I_{K,f},K^{\mathrm{ab}}/K)=\Gal(K^{ab}/K)$ so that we can ignore the archimedean component of $I_K$.} With respect to this $I_{K,f}$-action, the CM-point $P=[x_\rho, g]_\mathcal{K}$ has stabilizer
\begin{equation*}
K^\times \cdot \mathcal{K}_{g,\rho} \subseteq I_{K,f}, \ \mathcal{K}_{g,\rho}=\rho^{-1}(g\mathcal{K}g^{-1}).
\end{equation*}
We note that $\mathcal{K}_{g,\rho}$ is a compact open subgroup of $I_{K,f}$. This subgroup depends only on the Galois orbit of $P$ since $\mathcal{K}_{g,\rho}=\mathcal{K}_{\rho(t)g,\rho}$ for any $t\in I_{K,f}$ and $\mathcal{K}_{g,\rho}= \mathcal{K}_{hg,\mathrm{inn}(h) \circ \rho}$ for any $h \in \GL_2(\IQ)$. In particular, we may write $\mathcal{K}_P$ for $\mathcal{K}_{g,\rho}$ in the sequel, thereby associating a compact open subgroup of $I_{K,f}$ with each CM-point $P \in S_{\mathcal{K}}(\IC)$. 
The orbit of $P$ under $\Gal(K^{\mathrm{ab}}/K)$ can be identified with the double quotient 
\begin{equation} \label{equation::galoisorbit}
K^\times \backslash I_{K,f} / \mathcal{K}_{P},
\end{equation}
which is the set of complex points of a $0$-dimensional Shimura subvariety embedded into $S_\mathcal{K}$. In addition, the field of definition of $P$ is the fixed field of $(\mathcal{K}_{P},K^{\mathrm{ab}}/K) \subset \Gal(K^{\mathrm{ab}}/K)$.
 
Finally, we specialize to the modular curves $Y(1)$ and $Y_0(N)$ ($N\geq 1$). In this case, the identity $\GL_2(\IZhat)=\mathrm{Stab}_{\GL_2(\IAf)}(\IZhat^2)$ implies $\mathcal{K}_{g,\rho} = \mathrm{Stab}_{I_{K,f}}(\varphi^{-1}(g \cdot \IZhat^2))$.
It follows that $\mathcal{K}_{g,\rho}$ is the unit group of the ring
\begin{equation*}
\widehat{\mathcal{O}}_{g,\rho} = \{ r \in \IA_{K,f} \ | \  r \cdot \varphi^{-1}(g \cdot \IZhat^2) \subseteq \varphi^{-1}(g \cdot \IZhat^2) \},
\end{equation*}
which is easily seen to be an open subring of $\widehat{\mathcal{O}}_K$ containing $\IZhat$. It is hence the profinite completion of the imaginary quadratic order $\mathcal{O}_{g,\rho} = \widehat{\mathcal{O}}_{g,\rho} \cap \mathcal{O}_K$. In this way, we can associate with each CM-point $P=[x_\rho,g]_\mathcal{K} \in X(1)$ an order $\mathcal{O}_P \subseteq \mathcal{O}_K$, and the quotient (\ref{equation::galoisorbit}) can be identified with $\Pic(\mathcal{O}_P)= K^\times \backslash I_{K,f} / \widehat{\mathcal{O}}_P^\times$. For CM-points on $Y_0(N)$, a variant of the above starting from
\begin{equation*}
\widehat{\Gamma}_0(N)=\mathrm{Stab}_{\GL_2(\IA_{\IQ,f})}\left(\IZhat \cdot \begin{pmatrix} 1 \\ 0 \end{pmatrix} + \IZhat \cdot \begin{pmatrix} 0 \\ N \end{pmatrix}\right)
\end{equation*} 
allows to associate similarly a imaginary quadratic order with a CM-point $P$ on $Y_0(N)$; we again write $\mathcal{O}_P$ for this order. Although this is not relevant for us, let us remark that this order coincides for Heegner points with the one defined in \cite[Section 1]{Rosen2007}.

\subsection{Class number bounds} \label{subsection::classnumbers}

To each quadratic extension $K/\IQ$ there is associated a Dirichlet character $\chi_K(n)=(\disc(\mathcal{O}_K)/n)$ by means of Kronecker's symbol $(d/n)$; for details we refer to \cite[Section 5.2]{Cox1989}. For the associated Dirichlet $L$-function, the Siegel-Tatuzawa Theorem \cite{Tatuzawa1951} (see also \cite{Hoffstein1980/81}) states that there exists (at most) one imaginary quadratic field $K_\ast$ and an effectively computable constant $c_2(\varepsilon)$ such that
\begin{equation*}
L(1,\chi_K)> c_2(\varepsilon)\left|\disc(\mathcal{O}_K)\right|^{-\varepsilon}
\end{equation*}
for any other imaginary quadratic field $K \neq K_\ast$. Using the class number formula \cite[(15) on p.\ 49]{Davenport2000}, we infer that
\begin{equation*}
\# \Pic(\mathcal{O}_K) > c_3(\varepsilon) \left|\disc(\mathcal{O}_K)\right|^{1/2-\varepsilon}.
\end{equation*}
Even more, we can make use of \cite[Theorem 7.24]{Cox1989} to deduce the bound
\begin{equation} \label{equation::classnumberbound}
\# \Pic(\mathcal{O}) > c_4(\varepsilon) \left|\disc(\mathcal{O})\right|^{1/2-\varepsilon}
\end{equation}
for any imaginary quadratic order $\mathcal{O}$ not contained in $K_\ast$. Here, the constant $c_4(\varepsilon)$ is effectively computable.

\section{Intersections of class fields} \label{section::silverman}

\subsection{General class fields}\label{subsection::generalcase}

In this section, we establish Theorem \ref{theorem::rayclassfields}. Its Corollary \ref{corollary::ringclassfields} is deduced in the next section.

\begin{proof}[Proof of Theorem \ref{theorem::rayclassfields}] 
We first use the following  elementary argument repeatedly: If
\begin{equation*}
\xymatrix{0 \ar[r] & G_1 \ar[r] & G_0 \ar[r] & G_2 \ar[r] & 0}
\end{equation*}
is an exact sequence of (neither necessarily commutative nor finite) groups with $n_1$ (resp.\ $n_2$) annihilating $G_1$ (resp.\ $G_2$) then $G_0$ is annihilated by $n_1n_2$. Conversely, if $G_0$ is annihilated by $n_0$ then so are $G_1$ and $G_2$. We want to show that $\Gal(K_r^{\mathrm{ab}} \cap L/K_r k^{\mathrm{ab}})$ is annihilated by (\ref{equation::n}). As $\Gal(K K_r^{\mathrm{ab}} \cap L/K_r k^{\mathrm{ab}})$ surjects onto $\Gal(K_r^{\mathrm{ab}} \cap L/K_r k^{\mathrm{ab}})$, it suffices to show that $\Gal(K K_r^{\mathrm{ab}} \cap L/K_r k^{\mathrm{ab}})$ is annihilated by (\ref{equation::n}).\footnote{Although normality is not transitive in general, the normality of $K_i/k$ implies directly that $K_i^{\mathrm{ab}}/k$ is likewise normal. It follows that $KK_r^\mathrm{ab} \cap L / K_r k^{\mathrm{ab}}$ is indeed a normal extension.} Furthermore, by using the exact sequence
\begin{equation*}
\xymatrix{0 \ar[r] & \Gal(K K_r^{\mathrm{ab}} \cap L/K k^{\mathrm{ab}}) \ar[r] & \Gal(K K_r^{\mathrm{ab}} \cap L/K_r k^{\mathrm{ab}}) \ar[r] & \Gal(K k^{\mathrm{ab}}/K_r k^{\mathrm{ab}}) \ar[r] & 0}
\end{equation*}
we reduce to show that $\Gal(K K_r^{\mathrm{ab}} \cap L/K k^{\mathrm{ab}})$ is annihilated by $\lcm_{i \neq j}\{[K:K_i][K_j:k]\}$. In the remainder of this proof, all Galois groups encountered are abelian. 
By Lemma \ref{lemma::galois}, the (abelian) group $\Gal(K K_r^{\mathrm{ab}} \cap L / Kk^{\mathrm{ab}})$ is isomorphic to $\Coker(u)$ in the commutative diagram of abelian groups
\begin{equation*}
\xymatrix{ \Gal(K_1^{\mathrm{ab}}\cdots K_r^{\mathrm{ab}}/Kk^{\mathrm{ab}}) \ar[r]^<<<<<<<<{u} \ar[rd]_{w} & \Gal(K K_r^{\mathrm{ab}}/Kk^{\mathrm{ab}}) \ar[d]^{v} \times \Gal(L/Kk^{\mathrm{ab}}), \\ & \Gal(K_r^{\mathrm{ab}}/K_rk^{\mathrm{ab}}) \times \prod_{i=1}^{r-1} \Gal(K_i^{\mathrm{ab}}/K_ik^{\mathrm{ab}}),}
\end{equation*}
where the homomorphisms are the standard ones. As $v$ is injective, we can identify $\Coker(u)$ with a subgroup of $\Coker(w)$. In this way, $\Coker(w)$ contains an isomorphic copy of $\Gal(K K_r^{\mathrm{ab}} \cap L / Kk^{\mathrm{ab}})$. Since $K_1^{\mathrm{ab}}\cdots K_r^{\mathrm{ab}} \subset K^{\mathrm{ab}}$, $\Coker(w)$ is the same as the cokernel $\Coker(s)$ of
\begin{equation*}
s: \Gal(K^{\mathrm{ab}}/ Kk^{\mathrm{ab}}) \rightarrow \prod_{i=1}^{r} \Gal(K_i^{\mathrm{ab}}/K_ik^{\mathrm{ab}}), \ \sigma \mapsto (\sigma|_{K_1^{\mathrm{ab}}},\dots,\sigma|_{K_r^{\mathrm{ab}}}).
\end{equation*}
We claim that $\Coker(s)$ is annihilated by $\lcm_{i \neq j}\{[K:K_i][K_j:k]\}$.
To prove this, it suffices to show that for each $\sigma_1 \in \Gal(K_1^{\mathrm{ab}}/K_1k^{\mathrm{ab}})$ the element
\begin{equation*}
[K:K_1] \lcm_{j \neq 1}\{[K_j:k]\} \cdot \left(\sigma_1, \mathrm{id}_{K_2^{\mathrm{ab}}}, \dots,\mathrm{id}_{K_r^{\mathrm{ab}}}\right) \in \prod_{i=1}^{r} \Gal(K_i^{\mathrm{ab}}/K_ik^{\mathrm{ab}})
\end{equation*}
is already in the image of $s$. By (\ref{equation::classfieldtheory1}) from Section \ref{section::classfieldtheory}, we have a commutative diagram 
\begin{equation*}
\xymatrix{
N_{K/k}^{-1}(D_k) \ar[r]^<<<<<<<<<{t} \ar@{->>}[d]_{(\cdot,K^{\mathrm{ab}}/ K)} & \prod_{i=1}^{r} N_{K_i/k}^{-1}(D_k)
 \ar@{->>}[d]_{\prod_{i=1}^{r} (\cdot,K_i^{\mathrm{ab}}/ K_i)} 
\\
\Gal(K^{\mathrm{ab}}/ Kk^{\mathrm{ab}}) \ar[r]^<<<<{s} &
\prod_{i=1}^{r} \Gal(K_i^{\mathrm{ab}}/K_ik^{\mathrm{ab}}),
}
\end{equation*}
where the homomorphism $t$ sends $c \in N_{K/k}^{-1}(D_k)$ to $(N_{K/K_1}c,\dots, N_{K/K_r}c ) \in \prod_{i=1}^{r} N_{K_i/k}^{-1}(D_k)$. Choose an idele $c_1 \in N_{K_1/k}^{-1}(D_k) \subset C_{K_1}$ such that $(c_1,K_1^{\mathrm{ab}} | K_1)=\sigma_1$. We may consider $c_1$ also as an element of $C_K$ by using the canonical inclusion $C_{K_1} \hookrightarrow C_K$.
By normality, the Galois group $\Gal(K/k)$ acts on each $C_{K_i}$, $1 \leq i \leq r$. The idele classes $N_{K/K_i}(c_1)$, $2 \leq i \leq r$, are invariant under this action;  in fact, as $\Gal(K/K_i)$ is normal, each $\sigma \in \Gal(K/k)$ satisfies $\sigma \Gal(K/K_i)\sigma^{-1} = \Gal(K/K_i)$. Hence, for each $\sigma \in \Gal(K/K_1)$ we have
\begin{equation*}
\sigma(N_{K/K_i}(c_1)) = \prod_{\tau \in \Gal(K/K_i)}{\sigma \circ \tau(c_1)} = \sum_{\tau^\prime \in \Gal(K/K_i)}{\tau^\prime \circ \sigma(c_1)} = \sum_{\tau\in \Gal(K/K_i)}{\tau(c_1) } = N_{K/K_i}(c_1)\text{.}
\end{equation*}
Being a norm, $N_{K/K_i}(c_1)$ is also clearly $\Gal(K/K_i)$-invariant. From $K_1 \cap K_i = k$, we conclude the asserted $\Gal(K/k)$-invariance by \cite[Theorem VI.1.14]{Lang2002}. By Hilbert's Satz 90, it follows that $N_{K/K_i}(c_1) \in C_{K_i}^{\Gal(K_i/k)}=C_k$ (\cite[Theorem III.2.7]{Neukirch2013}) and consequently
\begin{equation*}
N_{K/K_i}(c_1)^{[K_i:k]} = N_{K_i/k}(N_{K/K_i}(c_1)) = N_{K_1/k}(c_1)^{[K:K_1]} \in D_k.
\end{equation*}
In other words, the automorphism $(N_{K/K_i}(c_1)^{[K_i:k]},k^{\mathrm{ab}} | k) \in \Gal(k^{\mathrm{ab}}/k)$ is the identity $\mathrm{id}_{k^{\mathrm{ab}}}$. Using (\ref{equation::classfieldtheory2}) in Section \ref{section::classfieldtheory}, we conclude that also
\begin{equation*}
(N_{K/K_i}(c_1)^{[K_i:k]},K_i^{\mathrm{ab}} | K_i)=\mathrm{Ver}(N_{K/K_i}(c_1)^{[K_i:k]},k^{\mathrm{ab}} | k) = \mathrm{Ver}(\mathrm{id}_{k^{\mathrm{ab}}}) = \mathrm{id}_{K_i^{\mathrm{ab}}}.
\end{equation*}
We infer that $\prod_{i=1}^{r} (\cdot,K_i^{\mathrm{ab}} | K_i)$ sends
\begin{equation*}
t(c_1)^{\lcm_{j \neq 1}\{[K_j:k]\}}= \left(c^{[K:K_1]}_1, N_{K/K_2}(c_1), \dots, N_{K/K_n}(c_1) \right)^{\lcm_{j \neq 1}\{[K_j:k]\}} \in \mathrm{image}(t)
\end{equation*}
into 
\begin{equation*}
[K:K_1] \lcm_{j \neq 1}\{[K_j:k]\} \cdot \left(\sigma_1, \mathrm{id}_{K_2^{ab}}, \dots,\mathrm{id}_{K_r^{ab}}\right) \in \mathrm{image}(s).
\end{equation*}
as claimed.
\end{proof}

The next lemma is completely elementary. Lacking references, we prove it here as a specialization of the results in \cite[Section VI.1]{Lang2002}. Note that the restriction to abelian extensions is necessary because the image of $\Psi$ is not normal in general.

\begin{lemma} \label{lemma::galois} Let $k$ be a field and $K_1/k$, $K_2/k$ two abelian Galois extensions contained in some common larger field. Then, the cokernel of the natural map
\begin{equation*}
\xymatrix{\Psi: \Gal(K_1K_2/k) \ar[r] & \Gal(K_1/k) \times \Gal(K_2/k),} \xymatrix{ \sigma \ar@{|->}[r] & (\sigma|_{K_1}, \sigma|_{K_2})}
\end{equation*}
is isomorphic to $\Gal(K_1\cap K_2/k)$.
\end{lemma}
\begin{proof} We establish that the image of $\Psi$ equals the kernel of the homomorphism
\begin{equation*}
\xymatrix{\Xi: \Gal(K_1/k) \times \Gal(K_2/k) \ar[r] & \Gal(K_1 \cap K_2 /k),} 
\xymatrix{ (\sigma_1,\sigma_2) \ar@{|->}[r] & \sigma_1|_{K_1\cap K_2}\sigma_2|_{K_1 \cap K_2}^{-1}.}
\end{equation*}
Evidently, $\image(\Psi)$ is contained in $\ker(\Xi)$. Assume first that both $K_1$ and $K_2$ are finite over $k$. As $\Xi$ is surjective, we have
\begin{equation*}
\# \ker(\Xi) = \frac{\# \Gal(K_1/k) \cdot \# \Gal(K_2/k)}{\# \Gal(K_1 \cap K_2/k)} = \# \Gal(K_1 K_2/k) = \# \image (\Psi),
\end{equation*}
where we used \cite[Theorem VI.1.12]{Lang2002} for the second equality. This numerical equality shows $\image(\Psi)=\ker(\Xi)$ as claimed. The general case follows by exhausting $K_1/k$ and $K_2/k$ with finite subextensions.
\end{proof}

\subsection{Ring class fields and transfer fields} \label{subsection::ringclassfields}
Let $K$ be a CM-field. This means that $K$ is an imaginary quadratic extension of a totally real number field $F$. To any $\mathcal{O}_F$-order $\mathcal{O} \subseteq \mathcal{O}_K$ is assigned its profinite completion $\widehat{\mathcal{O}}$, and its units $\widehat{\mathcal{O}}^\times$ embed diagonally into $I_{K,f} \subset I_K$. The ring class field $K[\mathcal{O}] \subset K^{\mathrm{rcf}}$ associated with $\mathcal{O}$ is the fixed field of $(\widehat{\mathcal{O}}^\times K^\times/K^\times,K^{\mathrm{ab}}/K) \subset \Gal(K^{\mathrm{ab}}/K)$. 
Using (\ref{equation::artinkernel}), we deduce
\begin{equation} \label{equation::picard2}
\Gal(K[\mathcal{O}]/K) = \frac{(I_{K,f}K^\times/K^\times,K^\mathrm{ab}/K)}{(\widehat{\mathcal{O}}^\times K^\times/K^\times,K^\mathrm{ab}/K)} = K^\times \backslash I_{K,f}/  \widehat{\mathcal{O}}^\times = \Pic(\mathcal{O}).
\end{equation}
Let $K^{\mathrm{rcf}} \subset K^{\mathrm{ab}}$ be the union of all ring class fields. Additionally, we can consider $C_F=I_F/F^\times=I_F K^\times/K^\times$ as a closed subgroup of $C_K$ and define the ``transfer field'' $K^{\mathrm{tf}} \subset K^{\mathrm{ab}}$ to be the fixed field of $(I_F K^\times/K^\times, K^{\mathrm{ab}}/K) \subset \Gal(K^{\mathrm{ab}}/K)$. The fields $K^{\mathrm{ab}}$ and $K^{\mathrm{rcf}}$ are intimately connected as the following lemma shows.

\begin{lemma} \label{lemma::transferfield} Let $F$ be a totally real field. For any totally imaginary quadratic extension $K/F$, $K^\mathrm{rcf}$ is a finite extension of $K^\mathrm{tf}$ and the Galois group $\Gal(K^\mathrm{rcf}/K^\mathrm{tf})$ is isomorphic to $\Pic(\mathcal{O}_F)$. 
\end{lemma}

If $F$ has class number one, then $K^{\mathrm{rcf}}=K^{\mathrm{tf}}$ for all totally imaginary extensions $K/F$. This is hence the case for the proofs of Theorems \ref{theorem::heegner} and \ref{theorem::andreoort} below.

\begin{proof} Each ideal $\mathfrak{f} \subseteq \mathcal{O}_F$ yields an $\mathcal{O}_F$-order $\mathcal{O}_F + \mathfrak{f} \mathcal{O}_K \subseteq \mathcal{O}_K$, and the intersection of all their profinite completions $\widehat{\mathcal{O}}_F + \mathfrak{f} \widehat{\mathcal{O}}_K \subseteq \widehat{\mathcal{O}}_K$ equals $\widehat{\mathcal{O}}_F$. In addition, the intersection $\mathcal{O}_1 \cap \mathcal{O}_2$ of any two $\mathcal{O}_F$-orders $\mathcal{O}_1, \mathcal{O}_2 \subseteq \mathcal{O}_K$ is again an $\mathcal{O}_F$-order. It follows that there is a descending chain of $\mathcal{O}_F$-orders $\mathcal{O}_i \subseteq \mathcal{O}_K$ such that $\widehat{\mathcal{O}} ^\times_F = \bigcap_i \widehat{\mathcal{O}}_i^\times$ and $K^{\mathrm{rcf}}= \bigcup_i K[\mathcal{O}_i]$.

We next claim the equality
\begin{equation} \label{equation::group}
(\widehat{\mathcal{O}} ^\times_F K^\times / K^\times, K^{\mathrm{ab}}/K) = 
\bigcap_i (\widehat{\mathcal{O}}_i^\times K^\times/K^\times, K^{\mathrm{ab}}/K).
\end{equation}
The inclusion ``$\subseteq$'' is trivial. For the other inclusion, we remark that by (\ref{equation::artinkernel}) the subgroup $\widehat{\mathcal{O}}_F^\times K^\times / K^\times \subset I_{K,f}K^\times/K^\times$ contains the kernel of $(\cdot, K^{\mathrm{ab}}/K)|_{I_{K,f}K^\times/K^\times}$. Hence, the identity (\ref{equation::group}) follows immediately from $\widehat{\mathcal{O}} ^\times_F = \bigcap_i \widehat{\mathcal{O}}_i^\times$.

We can now show that $K^{\mathrm{rcf}}$ is the fixed field of (\ref{equation::group}). Trivially, each element of $K^{\mathrm{rcf}}$ is fixed under this group. For the other direction, let $x \in K^{\mathrm{ab}}$ be fixed under (\ref{equation::group}) and let $\overline{K(x)}$ denote the normal closure of $K(x)$. This means that the intersection 
\begin{equation*}
\bigcap_i(\widehat{\mathcal{O}}_i^\times K^\times/K^\times, \overline{K(x)}/K) \subset \Gal(\overline{K(x)}/K)
\end{equation*}
contains only the identity $\mathrm{id}_{\overline{K(x)}}$. As $(\mathcal{O}_i)$ is a descending chain and $\Gal(\overline{K(x)}/K)$ is finite, we already have $(\widehat{\mathcal{O}}_i^\times K^\times/K^\times, \overline{K(x)}/K) = \{ \mathrm{id}_{\overline{K(x)}} \}$ for some order $\mathcal{O}_i$ so that $x \in K[\mathcal{O}_i]$.

From this alternative presentation of $K^{\mathrm{rcf}}$, it follows immediately that $K^{\mathrm{tf}} \subseteq K^{\mathrm{rcf}}$ and
\begin{equation} \label{equation::picard}
\Gal(K^{\mathrm{rcf}}/K^{\mathrm{tf}})=
\frac{(I_{F,f}K^\times/K^\times,K^\mathrm{ab}/K)}{(\widehat{\mathcal{O}}_F^\times K^\times/K^\times,K^\mathrm{ab}/K)}=
F^\times \backslash I_{F,f} / \widehat{\mathcal{O}}_F^\times = \Pic(\mathcal{O}_F);
\end{equation}
here, the injection $C_F \hookrightarrow C_K$ is used in addition to (\ref{equation::artinkernel}).
\end{proof}

A (profinite) group $G$ is called a generalized dihedral group with respect to a normal abelian subgroup $H$ of index $2$ if $G$ is the semidirect product $H \rtimes_{\iota} G/H$ with $\iota: G/H \rightarrow \Aut(H)$ sending the non-trivial element of $G/H$ to the inversion $H \rightarrow H$, $g \mapsto g^{-1}$. The following two lemmas seem to be folklore but we provide proofs here for lack of suitable references.

\begin{lemma} \label{lemma::dihedralgroup} Let $F$ be a totally real field and $K/F$ an imaginary quadratic extension of $F$. Then, $K^{\mathrm{tf}}/F$ is normal and the Galois group $\Gal(K^{\mathrm{tf}}/F)$ is a generalized dihedral group with respect to $\Gal(K^{\mathrm{tf}}/K)$.
\end{lemma}

\begin{proof} Let $\sigma$ be an automorphism of $\IQbar$ fixing $F$. Being a quadratic extension, $K/F$ is trivially normal and $\sigma(K)=K$. By (\ref{equation::classfieldtheory3}), the field $\sigma(K^{\mathrm{tf}})$ is the fixed field of $(\sigma(C_F), K^{\mathrm{ab}}/K)=(C_F, K^{\mathrm{ab}}/K)$ and thus $\sigma(K^{\mathrm{tf}})=K^{\mathrm{tf}}$. This shows that the extension $K^{\mathrm{tf}}/F$ is Galois. In addition, there is an exact sequence
\begin{equation*}
\xymatrix{
1 \ar[r] & \Gal(K^{\mathrm{tf}}/K) \ar[r] & \Gal(K^{\mathrm{tf}}/F) \ar[r]
& \Gal(K/F) \ar[r]
& 1.
}
\end{equation*}
Let $\iota \in \Gal(K^{\mathrm{tf}}/F)$ be a lifting of the non-trivial element of $\Gal(K/F)$. We have to show that $\iota \circ \sigma \circ \iota^{-1} = \sigma^{-1}$ for any $\sigma = (c, K^{\mathrm{tf}}/K)$, $c \in C_K$. By (\ref{equation::classfieldtheory3}),\begin{equation*}
(c,K^{\mathrm{tf}}/K) \circ
\left(
\iota \circ (c,K^{\mathrm{tf}}/K) \circ \iota^{-1} \right)= (c,K^{\mathrm{tf}}/K) \circ (\iota(c),K^{\mathrm{tf}}/K) =
(c \iota(c),K^{\mathrm{tf}}/K). 
\end{equation*}
As $c \iota(c) \in C_K^{\Gal(K/F)}=C_F$ by Hilbert's Satz 90, the automorphism $(c\iota(c),K^{\mathrm{tf}}/K)$ is the identity on $K^{\mathrm{tf}}$. In other words, we have $\sigma \circ (\iota \circ \sigma \circ \iota^{-1}) = \mathrm{id}_{K^{\mathrm{tf}}}$.
\end{proof}

The fact that $\Gal(K^{\mathrm{tf}}/F)$ is a generalized dihedral group has a substantial consequence due to the following elementary lemma.

\begin{lemma} \label{lemma::dihedralgroup2} Let $G$ be a generalized dihedral group with respect to a subgroup $H$ and let $H_0 \subset G$ be a normal subgroup such that the quotient $\pi: G \rightarrow G/H_0$ is abelian. Then, $H/(H \cap H_0)$ is annihilated by $2$.
\end{lemma}

\begin{proof} As $H/(H \cap H_0) \approx H H_0 / H_0$, the assertion means that $\pi(h)^2=1$ for any $h \in H$. Choose some $g_0 \in G \setminus H$. Then, $g_0 h g_0^{-1} = h^{-1}$ and hence $\pi(h)=\pi(h)^{-1}$ as $G/H_0$ is abelian. This already completes the proof.
\end{proof}

With the above lemmas at our disposal, we can deduce the announced corollary.

\begin{proof}[Proof of Corollary \ref{corollary::ringclassfields}] By Theorem \ref{theorem::rayclassfields}, $\mathrm{Gal}(K_r^\mathrm{ab} \cap \left( K_r \prod_{i=1}^{r-1}K_i^{\mathrm{ab}} \right) / K_r F^\mathrm{ab})$ and hence $\mathrm{Gal}(K_r^\mathrm{tf} \cap L / K_r^\mathrm{tf} \cap K_r F^\mathrm{ab} \cap L)$ is annihilated by 
\begin{equation*}
n= e(\Gal(K F^{\mathrm{ab}} /K_rF^{\mathrm{ab}})) \cdot \lcm_{i \neq j}\{[K:K_i][K_j:F]\}.
\end{equation*}
Since $K$ is a composite of quadratic extensions of $F$ the group $\Gal(K/F)$ is abelian of exponent $2$. Consequently, we have $K \subset F^{\mathrm{ab}}$ and hence $KF^{\mathrm{ab}} = F^\mathrm{ab}= K_rF^\mathrm{ab}$. We infer that $n=2^r$. From Lemma \ref{lemma::dihedralgroup} we know that $\Gal(K_r^\mathrm{tf}/F)$ is a generalized dihedral group with respect to its subgroup $\Gal(K_r^\mathrm{tf}/K_r)$. We may apply now Lemma \ref{lemma::dihedralgroup2} with $H_0= \Gal(K_r^{\mathrm{tf}}/K_r^\mathrm{tf} \cap F^\mathrm{ab} \cap L)$, obtaining that $\Gal(K_r^{\mathrm{tf}}\cap F^{\mathrm{ab}} \cap L/K_r)$ is annihilated by $2$. The assertion follows straightforwardly from the exact sequence
\begin{equation*}
\xymatrix{0 \ar[r] & \mathrm{Gal}(K_r^\mathrm{tf} \cap L / K_r^\mathrm{tf} \cap F^\mathrm{ab} \cap L) \ar[r] & \Gal(K_r^{\mathrm{tf}} \cap L/K_r) \ar[r] & \Gal(K_r^{\mathrm{tf}}\cap F^{\mathrm{ab}} \cap L/K_r) \ar[r] & 0.}
\end{equation*}
\end{proof}

\section{Proof of Theorem {\ref{theorem::heegner}}} \label{section::proof1}

Our proof improves upon the original one given by Rosen and Silverman \cite{Rosen2007}. We restate details from there as far as they are necessary to see that our Corollary \ref{corollary::ringclassfields} eliminates the restrictions of their proof effectively.

Assume given CM-points $P_1,\dots,P_r$ on $Y_0(N)$ associated with distinct CM-fields $K_i$ such that
\begin{equation} \label{equation::linearindependencerelation}
n_1 \cdot \pi(P_1) + n_2 \cdot \pi(P_2) + \cdots + n_r \cdot \pi(P_r) = 0, (n_1,\dots,n_r) \in \IZ^r \setminus \{ 0 \} \text{,}
\end{equation}
with respect to the group law on $E$. By using \cite[Theorem 1.5]{Nekovavr1999}, whose straightforward argument is completely effective (see the remarks in \cite[Section 3]{Nekovavr1999}) we may assume that at least two coefficients $n_i$ are non-zero. Assume hence that $n_r \neq 0$ and $K_r \neq K_\ast$ where $K_\ast$ is the (possibly non-existent) CM-field from Section \ref{subsection::classnumbers}. As explained in Section \ref{subsection::cmpoints}, an imaginary quadratic order $\mathcal{O}_i \subseteq \mathcal{O}_{K_i}$ is associated with each CM-point $P_i$. Furthermore, we assign a ring class field $K[\mathcal{O}_i]$ with each order $\mathcal{O}_i$ (see Section \ref{subsection::ringclassfields}).

Our next aim is to bound $\left|\disc(\mathcal{O}_r)\right|$. We do this by deriving two competing bounds on the size $\# \Pic(\mathcal{O}_r)$ of the Picard group $\Pic(\mathcal{O}_r)$. By (\ref{equation::picard2}), $\Pic(\mathcal{O}_r)$ is isomorphic to $\Gal(K_r[\mathcal{O}_r]/K_r)$. We consider the following diagram of abelian field extensions:
\begin{equation*}
\xymatrix{ & K_r(P_r) = K_r[\mathcal{O}_r] \ar@{-}[ld] \ar@{-}[ddd]\\ K_r(\pi(P_r)) \ar@{-}[d] & \\ K_r(n_r \cdot \pi(P_r)) \ar@{-}[rd] & \\ & K_r}
\end{equation*}
As both the group law on $E$ and the modular parameterization $\pi: X_0(N) \rightarrow E$ are defined over $\IQ$, it is a direct consequence of (\ref{equation::linearindependencerelation}) that
\begin{equation*}
K_r(n_r \cdot \pi(P_r)) \subseteq K_r[\mathcal{O}_r] \cap L \text{ where } L= K_r \prod_{i<r}K_i[\mathcal{O}_i]\text{.}
\end{equation*}
By our Corollary \ref{corollary::ringclassfields}, this implies that $\Gal(K_r(n_r \cdot \pi(P_r))/K_r)$ is annihilated by $2^{r+1}$. The degree $[K_r(\pi(P_r)):K_r(n_r \cdot \pi(P_r))]$ is bounded by some effectively computable constant $c_5(E)$; the proof is exactly the same as that of \cite[Theorem 11]{Rosen2007} noting that Serre's open image theorem is effective by \cite{Lombardo2015}. The degree $[K_r(P_r):K_r(\pi(P_r))]$ is less than $c_6(\pi)=\deg(\pi)$ due to an elementary argument (see \cite[Proposition 6]{Rosen2007}). The group $G=\Gal(K_r[\mathcal{O}_r]/K_r)$ is abelian and hence a direct product of cyclic groups $(\IZ/p^n\IZ)$ ($p$ a prime, $n$ a positive integer) by the elementary divisor theorem. Here, a factor of the form $(\IZ/2^n\IZ)$ can only occur if $$n\leq c_7(E,\pi,r) = (r+1) + \lfloor \log_2(c_5(E) c_6(\pi)) \rfloor.$$ Genus theory (cf.\ \cite[Proposition 6.3]{Zhang2005}) yields the bound
\begin{equation*}
\dim_{\IF_2} (\Pic(\mathcal{O}_r)[2]) \leq c_8(\varepsilon_1) \left|\disc(\mathcal{O}_r)\right|^{\varepsilon_1}
\end{equation*}
for some effective constant $c_8(\varepsilon_1)>0$. As $\# \Pic(\mathcal{O}_r)[2]$ is also the number of factors $(\IZ/2^n\IZ)$ occurring in $G$, it follows that
\begin{equation*}
\# \Pic(\mathcal{O}_r)^{\mathrm{even}} \leq \left(c_8(\varepsilon_1) \left|\disc(\mathcal{O}_r)\right|^{\varepsilon_1}\right)^{c_7(E,\pi,r)}\text{.}
\end{equation*}
Furthermore, we infer
\begin{equation*}
\# \Pic(\mathcal{O}_r)^{\mathrm{odd}} = [K_r[\mathcal{O}_r]:K_r]^{\mathrm{odd}} \leq [K_r[\mathcal{O}_r]:K_r(n_r \cdot \pi(P_r))] \leq c_5(E)c_6(\pi)
\end{equation*}
from the fact that $[K_r(n_r \cdot \pi(P_r)):K_r]$ is a power of $2$. By assumption, we have $K_r \neq K_\ast$ so that we may use (\ref{equation::classnumberbound}) with $\mathcal{O}=\mathcal{O}_r$. Combining this with the above two inequalities, we obtain
\begin{equation*}
c_4(\varepsilon_2) \left| \disc(\mathcal{O}_r)\right|^{1/2-\varepsilon_2} < \# \Pic(\mathcal{O}_r) \leq c_5(E)c_6(\pi)c_8(\varepsilon_1)^{c_7(E,\pi,r)} \left| \disc(\mathcal{O}_r)\right|^{\varepsilon_1c_7(E,\pi,r)}\text{.}
\end{equation*}
Choosing both $\epsilon_1$ and $\epsilon_2$ sufficiently small, we derive an upper bound on $\left|\disc(\mathcal{O}_r)\right|$.

\section{Proof of Theorem {\ref{theorem::andreoort}}}

Let $\pr_i: Y(1)^n \rightarrow Y(1)^{n-1}$ denote the projection omitting the $i$-th factor of $Y(1)^n$. We claim that the generic fiber dimension of $\pr_i|_X$ is zero; otherwise upper semicontinuity of the fiber dimension (\cite[Th\'eor\`eme 13.1.3]{EGA4III}) would imply that each fiber has dimension $\geq 1$. This is only possible if $\pi(X)=Y(1) \times \pr_i(X)$ for an appropriate coordinate permutation $\pi: Y(1)^n \rightarrow Y(1)^n$, which contradicts our assumptions. It follows that the closed points $P \in X$ such that $\dim(\pr_i^{-1}(\pr_i(P)) \cap X) \geq 1$ are all contained in a proper Zariski closed subset $Z_i \subsetneq X$. Working with explicit equations, we ascertain that the height and the degree of $Z^\prime = Z_1 \cup \cdots \cup Z_n$ can be effectively bounded in terms of the degree and height of $X$.

The closed set $Z^\prime$ constitutes a part of the exceptional set $Z$ whose existence is claimed in the theorem. To describe the other part, let $P=(P_1,\cdots,P_n) \in (X \setminus Z^\prime)(\IQbar)$ be a special point such that $K_i(P) \neq K_j(P)$ for all $i \neq j$. There exists some $i_0 \in \{ 1,\dots, n\}$ such that $K_{i_0}(P) \neq K_\ast$. Furthermore, denote by $\mathcal{O}_i$ the quadratic order associated with the CM-point $P_i$ as in Section \ref{subsection::cmpoints}. Recall that $K_i(P_i)=K_i[\mathcal{O}_i]$. There exists some effectively computable constant $c_9(X)$, depending only on the degree of $X$ and the degree of its field of definition, such that 
\begin{equation*}
[K_{i_0}[\mathcal{O}_{i_0}] \cdot L : L]\leq c_9(X) \text{ where } L = K_{i_0} \prod_{\substack{1\leq i \leq n \\ i \neq {i_0}}} K_i[\mathcal{O}_i].
\end{equation*}
Consider the following diagram of abelian field extensions:
\begin{equation*}
\xymatrix{ 
& K_{i_0}[\mathcal{O}_{i_0}] \cdot L \ar@{-}[ld] \ar@{-}[rd] & \\
K_{i_0}[\mathcal{O}_{i_0}] \ar@{-}[rd] & & L  \ar@{-}[ld] \\
& K_{i_0}[\mathcal{O}_{i_0}] \ar@{-}[d] \cap L \\
& K_{i_0} &
}
\end{equation*}
By \cite[Theorem VI.1.12]{Lang2002}, we have an equality $[K_{i_0}[\mathcal{O}_{i_0}] \cdot L:L] = [K_{i_0}[\mathcal{O}_{i_0}] : K_{i_0}[\mathcal{O}_{i_0}] \cap L]$. Furthermore, Corollary \ref{corollary::ringclassfields} implies that $\Gal(K_{i_0}[\mathcal{O}_{i_0}] \cap L/K_{i_0})$ is annihilated by $2^{n+1}$.

Analogous to the proof of Theorem \ref{theorem::heegner}, we can derive an upper bound for $\left| \disc (\mathcal{O}_{i_0}) \right|$ from this information. As there, genus theory yields
\begin{equation*}
\# \Pic(\mathcal{O}_{i_0})^{\mathrm{even}} \leq (\dim_{\IF_2}(\Pic(\mathcal{O}_{i_0})[2]))^{c_{10}(X)} \leq (c_8(\varepsilon_1)|\disc(\mathcal{O}_{i_0})|^{\varepsilon_1})^{c_{10}(X)}
\end{equation*}
with $c_{10}(X)=n+1 + \lfloor \log_2(c_9(X)) \rfloor$.
Furthermore, we have $\# \Pic(\mathcal{O}_{i_0})^{\mathrm{odd}} \leq c_9(X)$. Using the Siegel-Tatuzawa Theorem (\ref{equation::classnumberbound}), we obtain
\begin{equation*}
c_4(\varepsilon_2) \left| \disc(\mathcal{O}_{i_0})\right|^{1/2-\varepsilon_2} < \# \Pic(\mathcal{O}_{i_0}) \leq  c_8(\varepsilon_1)^{c_{10}(X)} c_9(X) \left| \disc(\mathcal{O}_{i_0})\right|^{\varepsilon_1c_{10}(X)}\text{.}
\end{equation*}
For $\varepsilon_1, \varepsilon_2 \rightarrow 0$, this leads to an effective upper bound $c_{11}(X)$ on $\left| \disc(\mathcal{O}_{i_0})\right|$. There exist only finitely many CM-points $Q_1,\dots,Q_m \in Y(1)$ whose associated imaginary quadratic order has discriminant less than $c_{11}(X)$. By assumption, the ${i}$-th coordinate function $x_{i}: Y(1)^n \rightarrow Y(1)$, $i \in \{ 1, \dots, n \}$, restricts to a non-constant regular function on $X$. This implies that 
\begin{equation*}
Z^{\prime \prime} = X \cap \left( \bigcup_{1 \leq i \leq n}\bigcup_{1 \leq j \leq m}( x_{i}^{-1}(Q_j)) \right)
\end{equation*}
is a proper algebraic subset of $X$. The above argument shows that the CM-points $P$ contradicting the assertion of the theorem with $Z = Z^\prime$ must be contained in $Z^{\prime \prime}$. Finally, the height and degree of $Z^{\prime \prime}$ can evidently be effectively bounded so that the assertion of the theorem is true for $Z = Z^\prime \cup Z^{\prime \prime}$.

\textbf{Acknowledgements:} The author thanks Yuri Bilu, Jonathan Pila, Harry Schmidt, and Gisbert W\"ustholz for their helpful remarks and for the discussions he had with them. In addition, he thanks both the Max Planck Institute for Mathematics and the Fields Institute for their supportive hospitality during the composition of this article. Finally, he acknowledges financial support by an Ambizione Grant of the Swiss National Science Foundation.

\bibliographystyle{plain}
\bibliography{../Bibliography/references}

\end{document}